\documentclass[a4paper,11pt]{article}
\usepackage[utf8]{inputenc}

\usepackage{boxedminipage}
\usepackage{amsfonts}
\usepackage{amsmath} 
\usepackage{amssymb}
\usepackage{graphicx}
\usepackage{amsthm}
\usepackage{t1enc}
\usepackage{subfig}
\usepackage{cancel}
 \usepackage{textcmds}
 \usepackage{mathabx}
 \usepackage{hyperref}
\usepackage{url}

\newtheorem{theorem}{Theorem}[section]
\newtheorem{lemma}[theorem]{Lemma}

\newtheorem{corollary}[theorem]{Corollary}
\newtheorem{definition}[theorem]{Definition}

\newtheorem*{theorem*}{Theorem}

\newcommand{\forceP}{\mathbb{P}}
\newcommand{\forceQ}{\mathbb{Q}}

\newcommand{\ZFC}{\mathsf{ZFC}}
\newcommand{\ZF}{\mathsf{ZF}}
\newcommand{\ZFP}{\mathsf{ZF}^-}

\newcommand{\PFA}{\mathsf{PFA}}
\newcommand{\BPFA}{\mathsf{BPFA}}

\newcommand{\PD}{\mathsf{PD}}

\def\undertilde#1{\mathord{\vtop{\ialign{##\crcr
$\hfil\displaystyle{#1}\hfil$\crcr\noalign{\kern1.5pt\nointerlineskip}
$\hfil\tilde{}\hfil$\crcr\noalign{\kern1.5pt}}}}}

\title{A Universe with large Continuum, global $\Sigma$-Uniformization and a projective Well-Order of its Reals}
\author{ Stefan Hoffelner$^{1}$\footnote{This research was funded in whole by the Austrian Science Fund (FWF) Grant-DOI 10.55776/P37228. }  }
\date{
    $^1$TU Wien \\
    \today }
\begin{document}

\maketitle

\begin{abstract}
We construct a model in which the continuum has size $\kappa$ for a regular cardinal $\kappa$ and in which the $\Sigma^1_n$-uniformization property holds simultaneously for every $n \ge 2$. Additionally this model has a $\Delta^1_3$-definable well-order of its reals.
\end{abstract}

\section{Introduction}

This work investigates universes whose definable sets of reals display a behaviour different than the one under the standard assumption of projective determinacy ($\mathsf{PD})$. Recall that the projective sets are the subsets of the reals which are definable over the structure $(H(\omega_1), \in )$ using first order formulas. For an arbitrary set of reals $A \subset \omega^{\omega}$, we can consider the two player game $G_A$, whose players, starting with player I, alternately play natural numbers in each turn, hence producing a real $r$ in an infinite amount of time. We say that player I wins the game, if the real $r$ is an element of the fixed set $A$. Otherwise player II wins. A strategy is a device which, given an arbitrary state of the game, tells the player which number to play next and a winning strategy of a player is a strategy which will always yield a win in the game for the player. Then $\PD$ is the assertion that, given any projective set $A \subset \omega^{\omega}$, either player I or player II has a winning strategy. 

It is well-known that $\PD$ yields a very rich and interesting theory of the projective subsets of the real numbers, deciding all the relevant questions, otherwise known to be independent of the axioms $\ZFC$. This, and the fact that both types of additional axioms set theorists typically use to strengthen the standard axioms, namely large cardinal axioms and forcing axioms, outright imply the truth of projective determinacy form very convincing evidence that $\PD$ is the right theory for the definable subsets of real numbers (see the very well-known \cite{Steel}, \cite{MS}, theorem 13.6.). 

Yet it is still interesting to investigate universes in which the definable subsets of reals act very differently than under the assumption of $\PD$. On the one hand, because set theorists always want to know which properties are still consistent with $\ZFC$ or $\ZFC$ and its strengthenings, and on the other hand as often fragments of the universe are used to investigate the real set theoretic universe itself. And these fragments will typically have a different theory of its local sets of reals than the surrounding universe.

This article deals with two notions which are ruled out by $\PD$. The first notion is the projectively definable well-order of the real numbers, and the second one is global $\Sigma$-uniformization. 
Given $A \subset \omega^{\omega} \times \omega^{\omega}$, we say that $f$ is a uniformization (or a uniformizing function) of
$A$ if $f$ is a partial function such that 
\[dom(f)= pr_1(A)= \{ x \in 2^{\omega} \, : \, \exists y ((x,y) \in A \} \] and the graph of $f$ is a subset of $A$.
\begin{definition}[Uniformization Property]
 We say that the projective pointclass $\Gamma \in \{\Sigma^1_n \mid n \in \omega\} \cup \{\Pi^1_n \mid n \in \omega\}$ has the uniformization
 property iff every $\Gamma$-set in the plane admits a uniformization 
 whose graph is in  $\Gamma$, i.e. the relation $(x,y) \in f$ is in $\Gamma$.
 \end{definition}
 
Due to Moschovakis (see \cite{Kechris} 39.9) $\PD$ implies that $\Pi^1_{2n+1}$ and $\Sigma^1_{2n+2}$-sets have the uniformization property. By the Martin-Steel theorem (see \cite{MS}, Theorem 13.6.), the assumption of infinitely many Woodin cardinals outright implies $\PD$, and hence large cardinals fully settle the behaviour of the uniformization property within the projective hierarchy. The size of the continuum can be arbitrarily changed, modulo the few constraints imposed by $\ZFC$, without altering  the zig-zag pattern of the uniformization property within the projective hierarchy.

The connection of $\PD$ with forcing axioms is established via Woodin's core model induction (see \cite{SchindlerSteel}). Under the assumption of the proper forcing axiom ($\PFA$), Schimmerling, Steel and Woodin showed  that $\PD$ is true (in fact much more is true, see \cite{Steel}, where it is shown that $\PFA$ implies that $\mathsf{AD}^{L(\mathbb{R} ) }$ holds), thus also under $\PFA$ the $\Pi^1_{2n+1}$-uniformization and $\Sigma^1_{2n}$-uniformization hold. Note that $\PFA$ implies that the continuum is $\aleph_2$.

Global $\Sigma$-uniformization is the statement that for all $n \ge 2$, $\Sigma^1_n$-sets have the uniformization property. It is well-known that $L$ satisfies global $\Sigma$-uniformization, which follows from the fact that the canonical well-order of the constructible reals is not only $\Delta^1_2$-definable but also a \emph{good} well-order. Recall that a $\Delta^1_n$-definable well-order $<$ of the reals is good if $<$ has ordertype $\omega_1$ and the relation $<_{I} \subset (\omega^{\omega})^2$ defined via
\[ x <_I y \Leftrightarrow \{(x)_n \, : \, n \in \omega\}= \{z \, : \, z < y\} \]
where $(x)_n$ is some fixed recursive partition of $x$ into $\omega$-many reals, is a $\Delta^1_n$-definable relation. Addison (see \cite{Addison}) observed that $\Delta^1_n$-definable well-orders imply $\Sigma^1_m$-uniformization for every $m \ge n$.

For a long time, good projective well-orders were the only tool to derive global $\Sigma$-uniformization, hence the question of whether global $\Sigma$-uniformization is consistent with the continuum being arbitrarily large remained open. Goal of this work is to prove that 
\begin{theorem}
Assume $V=L$ and let $\kappa$ be a regular cardinal. Then there is a generic extension of $L$ where $2^{\aleph_0}= \kappa$ and the $\Sigma^1_n$-uniformization property holds for every $n \ge 2$ simultaneously.
\end{theorem}
Adding an additional layer of coding forcings one can also obtain the next theorem which is the one we will prove in this article:
\begin{theorem}
Assume $V=L$ and let $\kappa$ be a regular cardinal. Then there is a generic extension of $L$ where $2^{\aleph_0}= \kappa$ and the $\Sigma^1_n$-uniformization property holds for every $n \ge 2$ simultaneously and the reals have a $\Delta^1_3$-definable well-order.
\end{theorem}
Universes with some interesting properties while having  a projectively definable well-order of the reals are a very active topic in current set theory (see e.g. \cite{Fischer Friedman Zdomskyy}, \cite{Fischer Friedman Khomskii} or \cite{Friedman Hoffelner} for a very small subset of the existing literature) and presumably started with L. Harrington's work in \cite{Harrington}.

This article employs the technique pioneered in \cite{BPFA and Global Uniformization} which uses infinitary logic and a copy mechanism to produce a method which forces global $\Sigma$-uniformization in the absence of a a good projective well-order. This technique is combined with the coding machinery from \cite{Ho2} and \cite{Ho4}, which a bit surprisingly, can also be used to find universes with a large continuum while keeping a robust coding mechanism.
There are naturally some differences though, the method in  \cite{BPFA and Global Uniformization} uses a countable support iteration. It is well-known that countable support iterations only allow a continuum of size $\le \aleph_2$, thus we are compelled to work with finite support iterations in order to obtain a universe with a large continuum.

We end with a short description of how the article is organized. Section 2 introduces the building blocks of our coding machine. It shows the existence of an $\omega_1$-length, independent sequence of Suslin trees which is definable in $L$, a result which must have been known before but we could not find a proof so we included the one from \cite{A Failure of Reduction in the Presence of Separation}. Section 3 defines coding forcings which can be used to switch a $\Sigma^1_3$-predicate from false to true. These coding forcings are as in \cite{Ho2} and \cite{Ho4}. Building on these codings we define a family of $\Sigma^1_n$-predicates for every $n \ge 3$ in section 3 which is inspired by \cite{BPFA and Global Uniformization}. These predicates can be controlled in a very precise way using iterations of the coding forcings from section 3, and will be used to produce the uniformizing functions we need when forcing the global $\Sigma$-uniformization. In section 5 we finally define the iteration which will produce the desired universe which shows our main theorem. That the iteration does what it should do is shown in section 6.

\section{Independent Suslin trees in $L$, almost disjoint coding}
The coding method of our choice utilizes Suslin trees, which can be generically destroyed in an independent way of each other. 
\begin{definition}
 Let $\vec{T} = (T_{\alpha} \, : \, \alpha < \kappa)$ be a sequence of Suslin trees. We say that the sequence is an 
 independent family of Suslin trees if for every finite set of pairwise distinct indices $e= \{e_0, e_1,...,e_n\} \subset \kappa$ the product $T_{e_0} \times T_{e_1} \times \cdot \cdot \cdot \times T_{e_n}$ 
 is a Suslin tree again.
\end{definition}
Note that an independent sequence of Suslin trees $\vec{T} = (T_{\alpha} \, : \, \alpha < \kappa)$ has the property that if $A \subset \kappa$ and we form 
$ \prod_{i \in A} T_i $
with finite support, where each $T_i$ denotes the forcing we obtain if we force with the nodes of the tree  as conditions using the tree order as the partial order, then in the resulting generic extension $V[G]$, for every $\alpha \notin A$, $V[G] \models `` T_{\alpha}$ is a Suslin tree$"$. To see this, we assume the opposite, namely there is an $\alpha \notin A$ such that $V[G]$ thinks that $T_{\alpha}$ is not Suslin anymore. We note that $V[G]$ is a generic extension of $V$ obtained with a forcing with the ccc (see \cite{Ho}, Lemma 51 for a more general argument of this type). So there is an $\aleph_1$-sized antichain of $T_{\alpha}$ in $V[G]$. But then the finitely supported product of the form $(\prod_{i \in A}  T_i )\times T_{\alpha}$ can not have the ccc, which is a contradiction to the  assumed independence of $\vec{T}$.

One can easily force the existence of independent sequences of Suslin trees with products of Jech's or Tennenbaum's forcing, or with just products of ordinary Cohen forcing. On the other hand independent sequences of length $\omega_1$ already exist in $L$. 

\begin{theorem}\label{DefinitionIndependentSequence}
Assume that $\aleph_1= \aleph_1^L$ and that $M$ is an uncountable model of $\ZF^{-} + ``\aleph_1$ exists$"$.  Then there is an independent sequence $\vec{S} = ( S_{\alpha} \mid \alpha < \omega_1)$ of Suslin trees in $L$ and the sequence $\vec{S}$ is uniformly $\Sigma_1 ( \{ \omega_1 \} )$-definable over $M$. To be more precise,
there is a $\Sigma_1$-formula $\phi$  with $\omega_1$ as the unique parameter, which does not depend on the model $M$, such that
the relation $\{ (t,\gamma) \mid \gamma < \omega_1 \land t \in T^{\gamma}  \}$, where $T^{\gamma}$ denotes the $\gamma$-th level of $T$, is definable over $M$ using $\phi$. 
\end{theorem}

\begin{proof}
We fix the canonical $\diamondsuit$-sequence $(a_{\alpha} \subset \alpha \mid \alpha< \omega_1)$ in $L$. 
Next we alter the usual construction of a Suslin tree from $\diamondsuit$ to construct an $\omega_1$-sequence of Suslin trees $\vec{T}= (T^{\alpha} \mid \alpha < \omega_1)$. We consider a partition of $\omega_1$ into $\omega_1$-many stationary  sets $\{ B_{\alpha } \mid \alpha < \omega_1 \}$ using the $\diamondsuit$-sequence. Hence we can assume that the partition is $\Sigma_1 ( \{ \omega_1 \} )$-definable over $L_{\omega_1} \subset M$.

If $\alpha$ is a limit stage, and $\beta$ is such that  $\alpha \in B_{\beta}$, then we want to construct the $\alpha+1$-th level of $T^{\beta}$, denoted by $T_{\alpha+1}^{\beta}$ under the assumption that $T_{\alpha}^{\beta}$ is already defined. 
First we assume that $\alpha$ is a not a limit point of $B_{\beta}$, then we define
$T_{\alpha+1}^{\beta}$ to be $T_{\alpha}^{\beta}$ and put infinitely many successors on each of the top nodes of $T^{\beta}_{\alpha}$.
Second we assume that $\alpha$ is a limit point of $B_{\beta}$. Then we define $T_{\alpha+1}^{\beta}$ as follows. We let $e$ be an element of $[\omega_1]^{<\omega}$ and we assume that for each $\delta \in e$, we have a tree $T^{\delta}_{\alpha}$ defined already.
We consider $a_{\alpha} \subset \alpha$. If $a_{\alpha}$ happens to be a maximal antichain $A$ in $\prod_{\gamma \in e} T^{\gamma}_{\alpha}$, then we seal that antichain off at level $\alpha+1$ for $ \prod_{\gamma \in e} T^{\gamma}_{\alpha+1}$, that is we chose $ \prod_{\gamma \in e} T^{\gamma}_{\alpha+1}$ in such a way that $A$ remains a maximal antichain in all further extensions of $ \prod_{\gamma \in e} T^{\gamma}_{\alpha+1}$.
Otherwise we just extend $T^{\beta}_{\alpha}$ via adding top nodes on countably many branches through $T^{\beta}_{\alpha}$.

We let $T^{\beta}: = \bigcup_{\alpha < \omega_1} T^{\beta}_{\alpha}$ and claim that 
$(T^{\beta} \mid \beta < \omega_1)$ is an independent sequence of Suslin trees in $L$.

Indeed, if $A \in L$ is an antichain in some $\prod_{\gamma \in e} T^{\gamma}$, then there is a constructible club $\alpha$ such that $A \cap \alpha$ is an antichain in $\prod_{\gamma \in e} T^{\gamma}_{\alpha}$. But then $A$ got sealed off in the next step of $\prod_{\gamma \in e} T^{\gamma}_{\alpha+1}$.

The definability of $\vec{S}$ comes from the fact that the canonical $\diamondsuit$-sequence in $L$ is $\Sigma_1 (\{ \omega_1 \})$-definable. We can use $L_{\omega_1}$ (which is $\Sigma_1( \{\omega_1 \} )$ to correctly define $\diamondsuit$ over it and consequentially $\vec{S}$ becomes definable over $L_{\omega_1}$ as well. The above considerations can be simulated correctly already in any transitive, uncountable $M$ which models $\ZF^{-}$, as it will compute $L_{\omega_1}$ correctly and the rest of the construction is performed inside the latter model. 
\end{proof}

We remark, that the above proof shows that countable levels of countable initial segments of $\vec{S}$ are already uniformly definable over countable transitive models $M$ of $\ZFP$ and $``\aleph_1$ exists and $\aleph_1= \aleph_1^L"$.
That is, if $M$ is as just stated and $\gamma, \eta<\aleph_1^M$ and $t \in T^{\gamma}$ such that $\phi(t,\gamma,\omega_1)$ is true in $L$, then $M \models \phi (t \upharpoonright \omega_1^M,\gamma,\omega_1)$ and vice versa. Indeed if $M$ is such a model, and $ M \cap \omega_1=\aleph_1^M= (\aleph_1^L)^M$, then
the canonical $\diamondsuit$-sequence can be computed correctly up to $\aleph_1^M$ inside $M$'s version of $L$. So the whole construction of $\vec{S}$ inside $L_{\omega_1}$, when repeated inside $L_{\aleph_1^M}$ will produce the claimed initial segment of $\vec{S}$ restricted to trees of height $\aleph_1^M$.

We briefly introduce the almost disjoint coding forcing due to R. Jensen and R. Solovay. We will identify subsets of $\omega$ with their characteristic function and will use the word reals for elements of $2^{\omega}$ and subsets of $\omega$ respectively.
Let $D=\{d_{\alpha} \, : \, \alpha < \aleph_1 \}$ be a family of almost disjoint subsets of $\omega$, i.e. a family such that if $r, s \in D$ then 
$r \cap s$ is finite. Let $X\subset  \omega$  be a set of ordinals. Then there 
is a ccc forcing, the almost disjoint coding $\mathbb{A}_D(X)$ which adds 
a new real $x$ which codes $X$ relative to the family $D$ in the following way
$$\alpha \in X \text{ if and only if } x \cap d_{\alpha} \text{ is finite.}$$
\begin{definition}\label{definitionadcoding}
 The almost disjoint coding $\mathbb{A}_D(X)$ relative to an almost disjoint family $D$ consists of
 conditions $(r, R) \in [\omega]^{<\omega} \times D^{<\omega}$ and
 $(s,S) < (r,R)$ holds if and only if
 \begin{enumerate}
  \item $r \subset s$ and $R \subset S$.
  \item If $\alpha \in X$ and $d_{\alpha} \in R$ then $r \cap d_{\alpha} = s \cap d_{\alpha}$.
 \end{enumerate}
\end{definition}
We shall briefly discuss the $L$-definable, $\aleph_1^L$-sized almost disjoint family of reals $D$  we will use throughout this article. The family $D$ is the canonical almost disjoint family one obtains when recursively adding the $<_L$-least real $x_{\beta}$ not yet chosen and replace it with $d_{\beta} \subset \omega$ where this $d_{\beta}$  is the real which codes the initial segments of $x_{\beta}$ using some recursive bijections between $\omega$ and $\omega^{<\omega}$. The definition of $D$ is uniform over any uncountable, transitive $\ZFP$-models $M$ with, as we can correctly compute $L$ up to $\aleph_1^L$ inside $M$ and then apply the above definition inside $L$'s version of $M$. Even more is true, if $M$ is a countable, transitive model of $\ZFP+``$ $\aleph_1$ exists and $\aleph_1=\aleph_1^L"$, then $M$ will compute $D \upharpoonright \omega_1^M$ in a correct way. The reason is again, that $M$ can define an initial segment of $L$ correctly which suffices to calculate $D \upharpoonright \omega_1^M$.

Last we state a short lemma which will be helpful when showing that our coding forcings work the way they should.
\begin{lemma}\label{a.d.coding preserves Suslin trees}
 Let $T$ be a Suslin tree and let $\mathbb{A}_D(X)$ be the almost disjoint coding which codes
 a subset $X$ of $\omega_1$ into a real with the help of an almost disjoint family
 of reals $D$ of size $\aleph_1$. Then $$\mathbb{A}_{D}(X) \Vdash_{} T \text{ is Suslin }$$
 holds.
\end{lemma}
\begin{proof}
 This is clear as $\mathbb{A}_{D}(X)$ has the Knaster property, thus the product $\mathbb{A}_{D}(X) \times T$ is ccc and $T$ must be Suslin in $V[{\mathbb{A}_{D}(X)}]$. 
\end{proof}

\section{Coding forcings}
We shall define the desired universe now. First we start with $L$ as our ground model.

In a first step we fix our regular cardinal $\kappa$ and generically add $\kappa$-many
$\omega_1$-Cohen subsets of $\omega_1$ with a countably supported product. We write $\mathbb{C} (\omega_1)$ for $\omega_1$-Cohen forcing, that is conditions are countable functions from some $\alpha < \omega_1$ to $2$ ordered by extension.

\[ \forceP^0= \prod_{\alpha < \kappa} \mathbb{C} (\omega_1). \]
Note that $\forceP^0$ is $\sigma$-closed and has the $\aleph_2$-c.c., thus cardinals are preserved.

In a second step we destroy every element of our $\omega_1$-sequence of independent sequence $\vec{S}$ of $L$-Suslin trees via generically adding a cofinal branch using a finitely supported product. 
So,  we let
\[ \forceP^1:= \prod_{\alpha <\omega_1} S_{\alpha}. \]
By independence of $\vec{S}$, $\forceP^1$ has the c.c.c and size $\aleph_1$, so $L^{\forceP^0\ast \forceP^1}$ and $L$ have the same cardinals and cofinalities.

We shall define the coding forcing now we will use throughout this article.
Our ground model is now $L^{\forceP^0 \ast \forceP^1}$ and fix a generic filter $g^0 \ast g^1$.  Let $g$ be an $\omega_1$-Cohen set added by $g^0$, i.e. assume that \[g^0= (g^0_{\alpha} \mid \alpha < \kappa )\] and there is an $\alpha < \kappa$ such that

\[g= g^0_{\alpha}. \]

Let $x$ be a real number within the model $L[g^0 \ast g^1]$. We are going to introduce a forcing notion, referred to as the ``coding forcing$"$, which utilizes a generically added subset of $\omega_1$, stemming from the generic filter $g^0$. The set $g$ can be used to identify the starting indices of $\omega$-blocks of the sequence of trees $\vec{S}$; within these blocks, we will encode $\omega_1$-branches through $\vec{S}$ in a manner that reflects the real $x$.

To formally specify the coding forcing $\forceP(g,x)$, we must first define several associated sets. Working within the constructible universe $L$, let's fix a bijection $\rho$ mapping countable subsets of $\omega_1$ (i.e., $[\omega_1]^{\omega}$) to ordinals below $\omega_1$. We then define $h^g := \{ \rho (g \cap \alpha) \mid \alpha < \omega_1 \}$. The function $\rho$ serves to convert the family of generically introduced $\omega_1$-Cohen subsets (potentially $\kappa$ many) into an almost disjoint family on $\omega_1$ of size $\kappa$. This means any two sets in the resulting family have only a countable intersection.

Let's enumerate $h^g$ as $(\alpha_i \mid i < \omega_1)$. For each $\alpha_i \in h^g$, we define a set of branches $B_{\alpha_i,x}$ corresponding to the pattern dictated by $x$ within the $\omega$-block of $\vec{S}$ starting at index $\omega \alpha_i$. Specifically:
\[ B_{\alpha_i,x} = \{ b_{\omega \alpha_i + 2n} \mid n \notin x \} \cup \{ b_{ \omega \alpha_i + 2n+ 1} \mid n \in x \} \]

Now, let $X$ be the union of all these branch sets: $X := \bigcup_{\alpha_i \in h^g} B_{\alpha_i,x}$. This set $X$ consists of $\omega_1$-many $\omega_1$-branches that pass through elements of the sequence $(S_{\omega \cdot \alpha_i + n} \mid n \in \omega, \alpha_i \in h^g)$. The structure of $X$ encodes the characteristic function of $x$ via the following $\Sigma_1(\{\omega_1,X\})$ formula:

\begin{align*}
\varphi(x) \equiv & \exists N (N \text{ is transitive}, |N| = \aleph_1, N \models \ZFP, X \in N \land \\ & N \models `` X \text{ defines a subset } h \subset \omega_1 \text{ such that } \\ & (\forall \beta \in h)(\forall n \in \omega) ( (n \in x \Rightarrow S_{\omega \beta + 2n + 1} \text{ possesses an } \omega_1\text{-branch}) \land \\ & \qquad \qquad \qquad \quad (n \notin x \Rightarrow S_{\omega \beta + 2n} \text{ possesses an } \omega_1\text{-branch}) ) )"
\end{align*}

It's worth noting that the implications in the formula could be strengthened to equivalences ($n \in x \Leftrightarrow S_{\omega \beta + 2n + 1}$ has an $\omega_1$-branch, and similarly for $n \notin x$), but this stronger form is not necessary for our purposes. Also, whenever we state "$S_{\delta}$ has an $\omega_1$-branch," we implicitly mean this is defined using the $\Sigma_1(\{\omega_1\})$-formula $\phi$ referenced in the proof of Lemma \ref{DefinitionIndependentSequence}, which allows defining the trees from $\vec{S}$ within the model $N$.

Our objective is to reformulate the set $X \subset \omega_1$ such that a localized version of the property $\varphi(x)$ holds even for appropriate countable transitive models. This adjustment is crucial because we seek a projective predicate, whereas $\varphi(x)$ as stated only functions reliably within $H(\omega_2)$.

To achieve this, we first select an ordinal $\beta$ of size $\aleph_1$ ensuring $X \in L_{\beta}[X]$ and $L_{\beta}[X] \models \ZFP$. Consequently, it must be that $L_{\beta}[X] \models \varphi(x)$. Next, we choose the $<_{L[X]}$-minimal club set $C \subset \omega_1$ belonging to $L[X]$, along with the $<_{L[X]}$-minimal sequence $(M_{\alpha} \mid \alpha \in C)$ of countable elementary submodels satisfying:
\[ \forall \alpha \in C (M_{\alpha} \prec L_{\beta}[X] \land M_{\alpha} \cap \omega_1 = \alpha) \]

Now, we introduce a set $Y \subset \omega_1$, with $Y \in L[X]$, designed to encode the pair $(C, X)$. The encoding scheme is as follows: The elements of $Y$ with odd indices will code the set $X$. Let $E(Y)$ represent the set of elements in $Y$ with even indices. If $\{c_{\alpha} \mid \alpha < \omega_1\}$ is the enumeration of the club $C$, then $E(Y)$ must satisfy these conditions:
\begin{enumerate}
    \item $E(Y) \cap \omega$ encodes a well-ordering of order type $c_0$.
    \item The intersection $E(Y) \cap [\omega, c_0)$ must be empty.
    \item For every $\beta$, the set $E(Y) \cap [c_{\beta}, c_{\beta} + \omega)$ encodes a well-ordering of order type $c_{\beta+1}$.
    \item For every $\beta$, the intersection $E(Y) \cap [c_{\beta}+\omega, c_{\beta+1})$ must be empty.
\end{enumerate}
The advantage of constructing this reshaped set $Y \in L[X]$ is captured by the following result. It demonstrates that certain countable transitive models of $\ZFP$, under mild extra conditions, are sufficient to recognize the branches encoding the characteristic function related to $x$.

\begin{lemma}
Consider an $\omega_1$-preserving outer model $\tilde{W}$ of $L[g^0][g^1]$. Let $X, C, Y \subset \omega_1$, $\gamma < \omega_1$, and $x \in L \cap \omega^{\omega}$ be as previously defined. For any countable transitive model $N \in \tilde{W}$ satisfying $\ZFP + ``\aleph_1 \text{ exists}"$, such that $\omega_1^N = (\omega_1^L)^N$ and $Y \cap \omega_1^N \in N$, the following holds:
\begin{align*}
 x \in N \land N \models \varphi(x)
\end{align*}
\end{lemma}

\begin{proof}
Let $N \in \tilde{W}$ be a countable transitive model satisfying the hypotheses $\omega_1^N = (\omega_1^L)^N$ and $Y \cap \omega_1^N \in N$.
First, we argue that $\omega_1^N$ must be an element of $C$. If not, $\omega_1^N$ would lie strictly between some $c_{\gamma}$ and $c_{\gamma+1}$. However, condition (3) in the definition of $Y$ implies that $N$ can discern that $c_{\gamma+1}$ is countable (since $E(Y) \cap [c_\gamma, c_\gamma + \omega)$ encodes a well-order of type $c_{\gamma+1}$, and this intersection is in $N$). This contradicts the fact that $\omega_1^N < c_{\gamma+1}$.

Since $\omega_1^N \in C$, consider the elementary submodel $M_{\omega_1^N} \prec L_{\beta}[X]$ from the sequence defined earlier. Let $\bar{M}$ be its transitive collapse.
Because $\omega_1^N \in C$, we know that $\bar{M}$ and $N$ must have the same $\omega_1$, i.e., $\omega_1^N = \omega_1^{\bar{M}}$. Furthermore, by elementarity ($M_{\omega_1^N} \prec L_{\beta}[X]$) and the fact that $L_{\beta}[X]$ satisfies "The minimal $\ZFP$-model $L_{\zeta}[X]$ verifies $\varphi(x)$", it follows that:
\[ M_{\omega_1^N} \models \text{``The least } \ZFP\text{-model } L_{\zeta}[X] \text{ witnesses that } \varphi(x) \text{ is true}." \]
After collapsing, this becomes:
\[ \bar{M} \models \text{``The least } \ZFP\text{-model } L_{\bar{\zeta}}[ X \cap \omega_1^{\bar{M}}] \text{ witnesses } \varphi(x)." \]
Since $N$ contains $Y \cap \omega_1^N$, it also contains $X \cap \omega_1^N$ (coded in the odd part). $N$ can therefore construct the model $L_{\bar{\zeta}}[X \cap \omega_1^N]$. Thus,
\[ N \models \text{``} L_{\bar{\zeta}}[X \cap \omega_1^{N}] \text{ witnesses that } \varphi(x) \text{ holds true."} \]
This implies $N \models \varphi(x)$, completing the proof.
\end{proof}

We now use the constructed set $Y \subset \omega_1, Y \in L[X]$ to finally define the forcing $\forceP(g,x)$. This forcing is specified as the almost disjoint coding forcing $\mathbb{A}_D(Y)$, performed relative to our fixed almost disjoint family of reals $D = \{d_{\alpha} \mid \alpha < \omega_1\}$ residing in $L$, coding $Y$ into a single real $r$.
Importantly, the definition of $\mathbb{A}_D(Y)$ depends solely on the subset $Y \subset \omega_1$ (which, in turn, depends only on $g$ and $x$) and the family $D$. It is independent of the ambient universe in which it's defined, provided that universe has the correct $\omega_1$ and contains the set $Y$. We also know from previous results that $\mathbb{A}_D(Y)$ preserves Suslin trees.

Let $G$ be a $\mathbb{A}_D(Y)$-generic filter over the model $L[g^0 \ast g^1]$. Let $r_Y$ be the generic real introduced by $G$. This real $r_Y$ encodes the set $Y \subset \omega_1$ through the following relationship:
\begin{align*}
\forall \alpha < \omega_1 (\alpha \in Y \Leftrightarrow r_Y \cap d_{\alpha} \text{ is finite})
\end{align*}
Due to the absolute nature of the definition of $D \in L$, this equivalence holds in all $\omega_1$-preserving outer models as well (in fact, it holds in all outer universes, although $Y$ might become countable in some, a detail we won't need).
The real $r_Y$ encapsulates sufficient information such that arbitrary countable $\ZFP$-models containing $r_Y$ and meeting a mild technical condition can verify that $\varphi(x)$ is true.

\begin{lemma}
Let $\tilde{W}$ be an outer universe of $L[g^0 \ast g^1][G]$, and let $x$ be the real fixed earlier. Within $\tilde{W}$, the real $r_Y$ possesses the following $\Pi^1_2(x)$-property:

\begin{align*}
({\ast}{\ast})_{r_Y} (x): \equiv & \text{ For every countable, transitive model } N \text{ of } \ZFP + ``\aleph_1 \text{ exists}" \\ & \text{ such that } \omega_1^N = (\omega_1^L)^N \text{ and } r_Y \in N, \text{ it holds that } \\ & N \models \varphi(x)
\end{align*}
\end{lemma}
\begin{proof}
Let's first assume $\tilde{W} = L[g^0 \ast g^1][G]$. The statement $(\ast \ast)_{r_Y}(x)$ is a $\Pi^1_2(r_Y)$-statement. If we can demonstrate its truth in $L[g^0 \ast g^1][G]$, Shoenfield absoluteness guarantees it will remain true in all outer models $\tilde{W}$.

Given a model $N$ as described in the hypothesis, since $\omega_1^N = (\omega_1^L)^N$ and the decoding process is absolute, $N$ can decode the set $Y \cap \omega_1^N$ from $r_Y$ using its own version of $D$ (which is simply $D \cap \omega_1^N$). Let $Z$ be the set coded by the odd entries of $Y \cap \omega_1^N$. Again, by the absoluteness of the decoding, $Z$ must be equal to $X \cap \omega_1^N$, where $X$ is the set from the previous lemma. Therefore, $N$ can determine $X \cap \omega_1^N$. As shown in the proof of the previous lemma, this allows $N$ to conclude:
\begin{align*}
 N \models \text{``The least } \ZFP \text{ model } L_{\zeta}[Z] \text{ witnesses that } \varphi(x) \text{ holds true."}
\end{align*}
Consequently, $N \models \varphi(x)$, as the lemma claims.
\end{proof}

To summarize: for a given real $x$ and a factor $g$ of $g^1$, the forcing $\forceP(g,x)$ is a proper forcing whose factors have size $\aleph_1$. It generically adds a real $r_Y$ which ensures that the $\Pi^1_2$-property $(\ast \ast)_{r_Y}(x)$ holds true.
More broadly, if $\tilde{W} \supseteq L[g^0 \ast g^1]$ is an outer universe, and there exists a real $r \in \tilde{W}$ that satisfies $(\ast \ast)_r(z)$ for some real $z \in \tilde{W}$, we say that $r$ witnesses that the real $z$ \emph{is written into $\vec{S}$}, or that $r$ witnesses $z$ \emph{is coded into $\vec{S}$}. If such a real $r'$ exists for $z$ in $\tilde{W}$, we simply say that $\tilde{W}$ thinks $z$ is coded into $\vec{S}$ or $\tilde{W}$ thinks $z$ is written into $\vec{S}$.

The statement ``$x$ is coded into $\vec{S}"$ corresponds to a $\Sigma^1_3(x)$-formula, which we denote by $\Phi^3(x)$. It can be expressed as follows, having the logical form $\exists r \forall M (\Delta^1_2(r,M,x) \rightarrow \Delta^1_2(r,M,x))$:
\begin{align*}
\Phi^3 (x) \equiv \exists r \forall M ( & ( M \text{ is countable and transitive, } M \models \ZFP + ``\aleph_1 \text{ exists}", \\ & \omega_1^M = (\omega_1^L)^M, \text{ and } r, x \in M ) \implies M \models \varphi(x) )
\end{align*}

The preceding lemma also has a converse. This means the projective and local property $(\ast \ast)_r(x)$ has implications for how certain inner models within the ambient universe perceive branches through $\vec{S}$.

\begin{lemma}
Let $\tilde{W} \supseteq L$ be an $\omega_1$-preserving outer model satisfying $\tilde{W} \models \ZFC$. Suppose $x, r$ are reals in $\tilde{W}$ such that $(\ast \ast)_r(x)$ holds true. Then any uncountable, transitive model $M \in \tilde{W}$ that contains $\{\omega_1, r\}$, satisfies $M \models \omega_1^M = \omega_1$, and $M \models \ZFP$, will also satisfy $M \models \varphi(x)$.
\end{lemma}

\begin{proof}
Suppose, for contradiction, that such an uncountable, transitive model $M$ exists but fails to satisfy $\varphi(x)$. By the Löwenheim-Skolem theorem, there exists a countable elementary submodel $N \prec M$ with $r \in N$. Let $\bar{N}$ be the transitive collapse of $N$. Then $\bar{N}$ would be a countable, transitive model satisfying the conditions for $(\ast \ast)_r(x)$ (since $M$ did, and these properties are elementary or preserved by collapse), but $\bar{N}$ would fail to satisfy $\varphi(x)$ (as $M$ didn't, and $\varphi(x)$ is $\Sigma_1$, its truth persists downwards via elementarity). This contradicts the assumption that $(\ast \ast)_r(x)$ holds for all such countable models.
\end{proof}

\begin{corollary}\label{codesdeterminerealworld}
Assume $\tilde{W}$ is an outer universe of $L$ sharing the same $\omega_1$. Suppose further that $r \in \tilde{W}$ is a real such that $\tilde{W} \models (\ast \ast)_r(x)$ for some real $x \in \tilde{W}$. Let $h \subset \omega_1$ be the set derived from the constructible sequence whose existence is guaranteed by $\varphi(x)$, representing the indices of $\omega$-blocks in $\vec{S}$ where the pattern for $x$ is encoded. If $\gamma \in h$, then within $\tilde{W}$, the following hold:
\begin{align*}
n \in x \implies L[r] \models ``S_{\omega \gamma + 2n+1} \text{ has an } \omega_1\text{-branch}."
\end{align*}
and
\begin{align*}
n \notin x \implies L[r] \models ``S_{\omega \gamma + 2n} \text{ has an } \omega_1\text{-branch}."
\end{align*}
\end{corollary}
\begin{proof}
First, by the previous lemma, since $L[r]$ is an inner model of $\tilde{W}$ containing $r$ and $\omega_1$, and satisfies $\ZFP$ and $\omega_1^{L[r]} = \omega_1$, we have $L[r] \models \varphi(x)$. The definition of the sequence $\vec{S}$ is absolute (using the formula $\phi$ from Lemma \ref{DefinitionIndependentSequence}'s proof) between $L[r]$ and $\tilde{W}$. The statement ``$S_{\beta}$ has an $\omega_1$-branch$"$ is $\Sigma_1(\{\omega_1\})$ and thus upwards absolute from $L[r]$ to $\tilde{W}$. Since $L[r] \models \varphi(x)$, it asserts the existence of these branches based on $x$. The corollary follows directly.
\end{proof}

The definition of the forcing $\forceP(g,x)$ exhibits a significant degree of absoluteness. Specifically, its definition is entirely independent of the universe in which it is computed, provided that universe contains the necessary set $Y \subset \omega_1$ (which encodes the relevant branches tied to $x$, the club $C$, etc., as defined above). We will leverage this property shortly.

Our strategy involves iteratively applying these coding forcings. The intention is to encode progressively more reals into the structure $\vec{S}$, thereby populating our target $\Sigma^1_3$-set, which comprises all reals coded into $\vec{S}$. It is essential that this iterative process does not inadvertently code reals that were not explicitly intended. The following result confirms this is the case.

Recall that $\Phi^3(x)$ denotes the $\Sigma^1_3$-statement "$x$ is coded into $\vec{S}$". Consider a finite support iteration $(\forceP_{\beta}, \dot{\forceQ}_{\beta} \mid \beta < \kappa)$ where each $\dot{\forceQ}_{\beta}$ is forced by $\forceP_{\beta}$ to be a coding forcing $\operatorname{Code}(\dot{x}_{\beta})$ for some $\forceP_{\beta}$-name for a real $\dot{x}_{\beta}$. Let $G$ be $\forceP_{\kappa}$-generic over $L[g^0][g^1]$. Within the generic extension $L[g^0][g^{1}][G]$, we define the set of reals \emph{intentionally coded} by the iteration $\forceP_{\kappa}^G$ as $\{\dot{x}_{\beta}^G \mid \beta < \kappa\}$.

\begin{lemma}\label{nounwantedcodes}
Let $\forceP \in L[g^0 \ast g^1]$ be a finite support iteration of coding forcings $\operatorname{Code}(x_\beta)$ of length $\kappa$. Let $G \subset \forceP$ be generic over $L[g^0 \ast g^1]$, and let $\{ x_{\beta} \mid \beta < \kappa\}$ be the set of reals intentionally coded by $\forceP^G$. Let $\Phi^3(v_0)$ be the formula defined earlier. Then, within the model $L[g^0][g^1][G]$, the set of reals $x$ satisfying $\Phi^3(x)$ is precisely $\{ x_{\beta} \mid \beta < \kappa\}$.
\end{lemma}
\begin{proof}
We work within $L[g^0][g^1][G]$. Assume, seeking a contradiction, that there exists a real $x$ such that $\Phi^3(x)$ holds, but $x$ is not among the intentionally coded reals, i.e., $x \notin \{x_{\beta} \mid \beta < \kappa\}$.
The truth of $\Phi^3(x)$ means there is a real $r$ witnessing it. Since $\forceP$ is a finite support iteration, any real in the extension $L[g^0][g^1][G]$ is actually present in an extension generated by a countable sub-iteration. Therefore, there exists a countable set $I \subset \kappa$ such that both $x$ and $r$ are elements of the intermediate model $L[g^0][g^1][G_I]$, where $G_I = G \cap \forceP_I$ and $\forceP_I$ is the sub-iteration restricted to indices in $I$.
Furthermore, coding each $x_i$ (for $i \in I$) involves countably many elements from $g^0$ and $g^1$. Thus, $x$ and $r$ actually reside in a model of the form $L[g^0 \upharpoonright K][g^1 \upharpoonright J][G_I]$ for some countable $K \subset \omega_1$ and some $J \subset \omega_1$ of size $\aleph_1$ whose complement also has size $\aleph_1$. Let $V = L[g^0 \upharpoonright K][g^1 \upharpoonright J][G_I]$. We conduct the remainder of the argument within this ZFC universe $V$, constructed inside $L[g^0][g^1][G]$.

Since $r$ witnesses $\Phi^3(x)$ in $V$, according to Corollary \ref{codesdeterminerealworld}, $r$ encodes information implying the existence of an unbounded set $h \subset \omega_1$ such that for all $\gamma \in h$:
\begin{align*}
n \in x \implies L[r] \models ``S_{\omega \gamma + 2n+1} \text{ has an } \omega_1\text{-branch}." \\
n \notin x \implies L[r] \models ``S_{\omega \gamma + 2n} \text{ has an } \omega_1\text{-branch}."
\end{align*}

Because $x$ is distinct from every $x_{\beta}$ with $\beta \in I$, there must exist $\aleph_1$-many ordinals $\alpha$ and some $n \in \omega$ where, without loss of generality, $x$ dictates a branch differently than the $x_\beta$'s. For instance, assume:
\[ L[r] \models ``S_{\omega \alpha + 2n+1} \text{ has an } \omega_1\text{-branch}," \]
but for every $\beta \in I$, if $r_{\beta}$ is the real witnessing $\Phi^3(x_{\beta})$, then:
\[ L[r_{\beta}] \models ``S_{\omega \alpha + 2n+1} \text{ does not have an } \omega_1\text{-branch}." \]

The set of indices corresponding to trees in $\vec{S}$ used by $g^1 \upharpoonright J$ forms a structure related to an almost disjoint family. We can choose such an $\alpha$ where the associated tree $S_{\omega \alpha + 2n+1}$ is not among those directly affected by $g^1 \upharpoonright J$ (which determines the components needed for $x$ and $r$). Let's fix such an $\alpha$.

We claim that no real $R$ exists within the model $V = L[g^0 \upharpoonright K][g^1 \upharpoonright J][G_I]$ such that $L[R] \models ``S_{\omega \alpha + 2n+1}$ has an $\omega_1$-branch$."$ This will yield the desired contradiction.

The claim holds because the forcing used to construct $V$ from $L$ did not explicitly add a branch through $S_{\omega \alpha + 2n+1}$. The forcing factors involved were either components of $g^0$ and $g^1$ (adding branches through other trees determined by $K$ and $J$) or the coding forcings $\operatorname{Code}(x_i)$ for $i \in I$. Coding forcings preserve Suslin trees, and the sequence $\vec{S}$ is independent. Therefore, within the model $V$, the tree $S_{\omega \alpha + 2n+1}$ remains a Suslin tree (i.e., it has no $\omega_1$-branch).
However, if there were a real $R \in V$ such that $L[R] \models ``S_{\omega \alpha + 2n+1}$ has an $\omega_1$-branch$,"$ then by upward absoluteness (as the statement is $\Sigma_1(\{\omega_1\})$), an $\omega_1$-branch through $S_{\omega \alpha + 2n+1}$ would have to exist in $V$ itself. This contradicts the fact that $S_{\omega \alpha + 2n+1}$ remains Suslin in $V$.
This contradiction refutes the initial assumption that an unintended real $x$ could satisfy $\Phi^3(x)$.
\end{proof}

As a final observation relevant for later use: the argument in the preceding proof holds even if we intersperse Cohen forcing $\mathbb{C}$ (or any other Suslin tree preserving, ccc forcing notion) within the finite support iteration. That is, the lemma remains valid for iterations whose factors are either coding forcings or standard Cohen forcing. This flexibility will be useful later, particularly for achieving goals related to increasing the size of the continuum.

\section{Suitable $\Sigma^1_n$-predicates}

We shall use the $\Sigma^1_3$-predicate ``being coded into $\vec{S}"$ (we will often write just ``being coded$"$ for the latter) to form suitable $\Sigma^1_n$-predicates $\Phi^n$ for every $n \in \omega$. These predicates share the following properties:
\begin{enumerate}
\item $L \models \forall x  \lnot(\Phi^n (x))$
\item For every real $x \in L$, there is an iteration of coding forcings $\operatorname{Code}^n (x) \in L$ such that after forcing with it, $L^{\operatorname{Code}^n (x)} \models \Phi^n (x)$, and for every real $y \ne x$, $L^{\operatorname{Code}^n (x)} \models \lnot \Phi^n(y)$.
\end{enumerate}
Most importantly, these properties remain true even when iterating the (iterations of coding forcings
$\operatorname{Code}^n (x_i)$ for a sequence of (names of) reals.

The predicates $\Phi^n(x)$ will be defined now. In the following we let
$(x,y)$ denote a real $z$ which recursively codes the pair of reals consisting  of $x$ and $y$. Likewise $(x_0,...,x_n)$ is defined.

\begin{itemize}
\item $\Phi^3 (x,y,m) \equiv \exists a_0 ( (x,y,m,a_0)$ is coded into $\vec{S})$.
\item $\Phi^4(x,y,m) \equiv \exists a_0 \forall a_1 ( (x,y,m,a_0,a_1)$ is not coded into $\vec{S} )$.
\item $\Phi^5(x,y,m) \equiv \exists a_0 \forall a_1 \exists a_2 ( (x,y,m,a_0, a_1,a_2)$ is coded into $\vec{S} )$.
\item $\Phi^6 (x,y,m) \equiv \exists a_0 \forall a_1 \exists a_2 \forall a_3 ( (x,y,m,a_0,a_1, a_2,a_3)$ is not coded into $\vec{S})$.
\item ...
\item ...
\item $\Phi^{2n} (x,y,m) \equiv \exists a_0 \forall a_1...\forall a_{2n-3} (( x,y,m,a_0,...,a_{2n-3} )$ is not coded into $\vec{S} )$.
\item $\Phi^{2n+1} (x,y,m) \equiv \exists a_0 \forall a_{1}...\exists a_{2n-2} (( x,y,m,a_0,...,a_{2n-2})$ is coded into $\vec{S} )$.
\item ...
\item...

\end{itemize}
Each predicate $\Phi^n$ is exactly $\Sigma^1_n$.
In the choice of our $\Sigma^1_n$-formulas $\Phi^n(x)$, we encounter again a periodicity phenomenon, that is two different cases depending on $n\in \omega$ being even or odd, a theme which is pervasive in this area.
It is clear that for each predicate $\Phi^n$ and each given real $x$ there is a way to create a universe in which $\Phi^n (x)$ becomes true using our coding forcings. We just need to iterate the relevant coding forcings using countable support. For $n=3$ we just need one coding forcing, for $n \ge3$ our iteration will have inaccessible length in order to catch our tail. As shown already, our coding method allows us to exactly code the tuple of reals we want to code, without accidentally adding some unwanted information. Thus the next lemma is straightforward to prove, so we just state it.
\begin{lemma}
Assume the existence of an inaccessible limit of inaccessible cardinals in $L$. Let $n \in \omega$ and let $x$ be a real in our ground model $L$. Then there is a forcing $\operatorname{Code}^n (x)$ which is $\vec S$-coding such that if $G \subset \operatorname{Code}^n (x)$ is generic,
$L[G]$ will satisfy $\Phi^n(x)$ and for every $y \ne x$, $L[G] \models \lnot \Phi^n(y)$.
This property can be iterated, that is it remains true if we replace $L$ with $L[G]$ in the above.

\end{lemma}

\section{Defining the desired forcing iteration}

We shall work towards finding the right iteration which will eventually prove the main theorem. The technique to force the global $\Sigma$-uniformzation property is a sort of a copying mechanism where we use the coding forcings to code up infinite conjunctions of projective formulas. The construction is pioneered in \cite{BPFA and Global Uniformization} and applied there in a different context using a different coding technique. It is flexible enough to be applicable in our context with our coding machinery as well.

\subsection{Forcing $\Sigma^1_3$-uniformization}
As a first step, we shall consider the problem of forcing the $\Sigma^1_3$-uniformization property over our universe $L[g^0 \ast g^1]$. We fix some notation first. We write 
\[g^1= (g^1_{\alpha} \mid \alpha < \kappa ) \]
and 
\[ \operatorname{Code} (x,\eta) := \forceP( g^1_{\eta}, x), \]
in other words $\operatorname{Code} (x,\eta)$ will code the real $x$ into $\vec{S}$ at the $\eta$-th element of our almost disjoint family we generically created with $\forceP^1$ and the bijection $\rho: [\omega_1]^{\omega} \rightarrow \omega_1$ in $L[g^0 \ast g^1]$.

There is a very easy strategy to force the $\Sigma^1_3$-uniformization property over $L[g^0 \ast g^1]$: we pick some bookkeeping function $F \in L[g^0 \ast g^1]$ which should list all the (names of) reals in our iteration and at each stage $\beta$, under the assumption that we created already $L[g^0 \ast g^1][G_{\beta}]$, if $F(\beta)$ lists a real $x\in L[g^0 \ast g^1][G_{\beta}]$ and a natural number $k$, we ask whether there is a real $y$ such that
\[ L[g^0 \ast g^1][G_{\beta}] \models \varphi_k (x,y) \]
holds, where $\varphi_k$ is the $k$-th $\Sigma^1_3$-formula. If so, we pick the least such $y$ (in some fixed well-order), and let the value of the desired $\Sigma^1_3$-uniformizing $f_k(x)$ to be $y$. Additionally we force with $\operatorname{Code} ((x,y,a_0,k), \beta)$, for $a_0$ an arbitrary real over $L[g^0 \ast g^1][G_{\beta}]$ and obtain $L[g^0 \ast g^1][G_{\beta+1}]$. 
The resulting universe $L[g^0 \ast g^1][G_{\beta+1}]$ will satisfy that $\Phi^3(x,y,k)$ is true, whereas $\Phi^3(x,y',k)$ is not true for each $y' \ne y$. Moreover, because of upwards absoluteness of $\Sigma^1_3$-formulas, this property will remain true in all further generic extensions we create, as long as we do not force with a forcing of the form $\operatorname{Code} ((x,y,a,k), \eta)$ ever again, where $a$ is a real and $\eta$ is some ordinal.

If we repeat this reasoning for all $\Sigma^1_3$-formulas and all reals $x$, and iterate long enough in order to catch our tail, the final model $L[g^0 \ast g^1][G]$ will 
satisfy the $\Sigma^1_3$-uniformization property via
\[f_k (x)=y \Leftrightarrow \Phi^3 (x,y,k). \]
To summarize, the easy strategy to force $\Sigma^1_3$-uniformization is to consider at each step some $x$-section of some $\Sigma^1_3$-set $A_k \subset \omega^{\omega} \times \omega^{\omega}$, and if non-empty, pick the least $y$ for which $A_k(x,y)$ is true. Then force to make $\Phi^3(x,y,k)$ true and repeat.

Based forcings use this strategy for $\Sigma^1_3$-sets, while putting no constraints on $\Sigma^1_n$-sets for $n > 3$.

\begin{definition}
Assume that $\kappa$ is a regular cardinal in $L$, let $\lambda < \kappa$ and let $F: \lambda \rightarrow H(\kappa)$ be a bookkeeping function. We say that an iteration $(\forceP_{\beta}, \dot{\forceQ}_{\beta} \mid \beta < \lambda)$ is (0-)based with respect to $F$ if the iteration is defined inductively via the following rules:
Assume that $F(\beta)= (\dot{x}, \dot{k})$ where $\dot{x}$ is a $\forceP_{\beta}$-name of a real
and $\dot{k}$ a $\forceP_{\beta}$-name of a natural number which itself is the G\"odelnumber of a $\Sigma^1_3$-formula. Also assume that $V[G_{\beta}] \models \exists y (\varphi_{k}(x,y))$ and let $\dot{y}$ be the $<$-least name of such a real in some fixed well-order of $H(\kappa)$.
Then let \[ \dot{\forceQ}_{\kappa}^{G_{\beta}} := \operatorname{Code} ((x,y,k), \eta), \] where $\eta$ is the least ordinal such that \[ \forceP_{\beta} \Vdash ``\text{there is no real z such that } \operatorname{Code} (z, \check{\eta} ) \text{ is a factor of } \forceP_{\beta}" ,\] and $(x,y,k)$ should denote a real coding the tuple consisting of $x,y$ and $k$.
If $V[G_{\beta}] \models \exists y (\phi_{k}(x,y))$ is not true use the trivial forcing.

If $F(\beta)= (\dot{k}, \dot{x},\dot{y}, \dot{\iota} )$, where $\dot{k}$ is a $\forceP_{\beta}$-name of a natural number, which itself is the G\"odelnumber of a $\Sigma^1_m$-formula, $m > 3$,  $\dot{x},\dot{y}$ are names of reals, and finally $\dot{\iota}$ is a name of an element of $2= \{0,1 \}$, then we let
\[ \dot{\forceQ}_{\beta}^{G_{\beta}} := \operatorname{Code} ((x,y,k),\eta) \]
for $\eta$ the least inaccessible above $| \forceP_{\beta} |$, provided $\iota^{G_{\beta}}=1$.

If $\iota^{G_{\beta}}=0$, we force with the trivial forcing instead.

\end{definition}

\subsection{Strategy to obtain global $\Sigma$-uniformization}
We shall describe the underlying idea to force the global $\Sigma$-uniformization.
Before we start, we fix some notation which should simplify things a bit. First, whenever we use an iteration using coding forcings $\operatorname{Code} (x,\eta)$ over the  universe $L[g^0\ast g^1]$, and we want to define the $\beta$-th factor $\dot{\forceQ}_{\beta}^{G_{\beta}}$  of our iteration namely a forcing of the form $\operatorname{Code} (z,\eta)$, we typically will not specify the ordinal $\eta$ anymore as it will always be the least ordinal $\eta$ such that 
\[ L[g^0][g^1][G_{\beta}]  \Vdash ``\text{there is no real z such that } \operatorname{Code} (z, \check{\eta} ) \text{ is a factor of } \forceP_{\beta}" .\] We will from now on just write $\operatorname{Code} (x)$, where we actually should have written $\operatorname{Code}(x,\eta)$. Second, we will typically code reals $z$ which in fact are recursive codes for tuples of reals $(x,y,m, a_0,a_1,...,a_n)$, for which we just write
$\operatorname{Code} (x,y,m,a_0,...,a_n)$, instead of $\operatorname{Code} ((x,y,m,a_0,...,a_n))$. Now back to the discussion of the strategy we aim to use in our desired forcing.

The definition of the factors of the iteration will depend on whether the formula
$\varphi_m$ we consider at our current stage is in $\Sigma^1_{2n}$ or in $\Sigma^1_{2n+1}$ where $n\ge 2$. We start with the case where $\varphi_m$ appears first on an odd level of the projective hierarchy.
Assume that $$\varphi_m \equiv \exists a_0 \forall a_1 \exists a_2 ... \exists a_{2n-2} \psi( x,y,a_0,a_1,... a_{2n} )$$  is a $\Sigma^1_{2n+1}$-formula (where $\psi(x,y,a_0,...a_{2n})$ is a $\Pi^1_2$-formula quantifying over the two remaining bounded variables $a_{2n-1}$ and $a_{2n}$) and  $x$ is a real. We want to find a value for the uniformization function for $\varphi_m$ at the $x$-section of $A_m$, the latter being the projective set in the plane defined by $\varphi_m$.

To start, we list all triples of reals $$ ((x,y^0,a_0^0) , (x, y^1, a_0^1), (x,y^2,a_0^2) ,...)$$ according to some fixed well-order $<$. To rule out a degenerate case, we assume that for any real $x$, the first such triple in our wellorder is always $(x,0,0)$, i.e. for all $x$, $a_0=y_0=0$. We decide that if $\forall a_1 \exists a_2...\psi (x,0,0,a_1,...)$ is true, then our uniformizing function for $\varphi_m$ at $x$ should have value $0$. Note that $\forall a_1 \exists a_2 ... \psi(x,0,0,a_1,...)$  is $\Pi^1_{2n}$.

Otherwise we work under the $\Sigma^1_{2n}$-assumption 
$\exists a_1 \forall a_2... \lnot \psi (x,0,0,a_1,...)$ and continue.
  There will be a $<$-least triple $(x,y^{\alpha}, a_0^{\alpha} )$, $\alpha \ge 1$, for which $\forall a_1 \exists a_2... \exists a_{2n-2} \psi (x,y^{\alpha},a^{\alpha}_0, a_1,a_2...)$ is true, otherwise the $x$-section would be empty and there is nothing to uniformize.

The goal will be to set up the iteration in such a way that all triples $(x,y^{\beta} ,a_0^{\beta})$, $\beta > \alpha$ will satisfy the following formula, which is $\Pi^1_{2n+1}$ in the parameters $( x , y^{\beta}, a^{\beta}_0 ,m)$ as is readily checked:
\begin{align*}
\forall a_{1} \exists a_{2} \forall a_{3}... \exists a_{2n-2} (( x,y^{\beta} ,m, a_0^{\beta},a_1,a_2...,a_{2n-2} )\text{ is not coded into $\vec{S} )$.}
\end{align*}
At the same time the definition of the iteration will ensure that for every $\beta \le \alpha$ our tuple $(x,y^{\beta}, a^{\beta}_0)$ will satisfy 
\begin{align*}
\exists a_{1} \forall a_{2} \exists a_{3}... \forall a_{2n-2} (( x,y^{\beta} ,m, a_0^{\beta},a_1,a_2...,a_{2n-2} )\text{ is  coded into $\vec{S} )$.}
\end{align*}
Provided we succeed, the pair $(x,y^{\alpha})$ will then be the unique solution to the following formula, which is $\Sigma^1_{2n+1}$, and which shall be the defining formula for our uniformizing function:
\begin{align*}
\sigma_{\text{odd}}(x,y,m) \equiv \exists a_0 ( \forall a_1 \exists a_2...\psi (x,y,a_0,a_1,...) \land \\
\lnot ( \forall a_1 \exists a_2... ( ( x,y ,m, a_0 ,a_1 ,a_2...) \text{ is not coded into $\vec{S}$}) )
\end{align*}
Indeed, for all $\beta> \alpha$, $(x,y^{\beta},a^{\beta}_0)$ can not satisfy the second subformula of $\Psi$ whereas for all
$\beta< \alpha$ (note here that $\alpha \ge 1)$, $(x,y^{\beta} , a_0^{\beta})$ can not satisfy the first subformula, as $(x,y^{\alpha}, a_0^{\alpha} )$ is the least such triple.

If we assume that $\varphi_m$ is on an even level of the projective hierarchy we will define things in a dual way to the odd case. 
Assume that $\varphi_m \equiv \exists a_0 \forall a_1 \exists a_2 ... \forall a_{2n-3} \psi( x,y,a_0,a_1,... a_{2n-3} )$ is a $\Sigma^1_{2n}$-formula and  $x$ is a real. We want to find a value for the uniformization function for $\varphi_m$ at the $x$-section.

Again, we list all triples of reals $ ((x,y^0,a_0^0) , (x, y^1, a_0^1), (x,y^2,a_0^2) ,...)$ according to some fixed well-order $<$ with $a_0=y_0=0$. We assume that $\exists a_1 \forall a_2... \lnot \psi(x,0,0,a_0,a_1,...)$ (otherwise $0$ would be the value for the uniformizing function at $x$ for $\varphi_m$) so  there will be a $<$-least triple $(x,y^{\alpha}, a_0^{\alpha} )$, $\alpha \ge 1$, for which $\forall a_1 \exists a_2... \forall a_{2n-3} \psi (x,y^{\alpha},a^{\alpha}_0, a_1,a_2...)$ is true, otherwise the $x$-section would be empty and there is nothing to uniformize.

The goal will be to set up the iteration in such a way that all triples $(x,y^{\beta} ,a_0^{\beta})$, $\beta > \alpha$ will satisfy the following formula, which is $\Pi^1_{2n}$ in the parameters $( x , y^{\beta}, a^{\beta}_0,m)$ as is readily checked:
\begin{align*}
\forall a_{1} \exists a_{2} \forall a_{3}... \exists a_{2n-3} (( x,y^{\beta} ,m, a_0^{\beta},a_1,a_2...,a_{2n-3} )\text{ is coded into $\vec{S} )$.}
\end{align*}
At the same time the definition of the iteration will ensure that for every $\beta  \le \alpha$  $(x,y^{\beta}, a^{\beta}_0)$ will satisfy 
\begin{align*}
\exists a_{1} \forall a_{2} \exists a_{3}... \exists a_{2n-3} (( x,y^{\alpha} ,m, a_0^{\alpha},a_1,a_2...,a_{2n-3} )\text{ is not  coded into $\vec{S} )$.}
\end{align*}
Provided we succeed, the pair $(x,y^{\alpha})$ will then be the unique solution to the following formula, which is $\Sigma^1_{2n}$, and which shall be the defining formula for our uniformizing function:
\begin{align*}
\sigma_{\text{even}} (x,y,m) \equiv \exists a_0 ( \forall a_1 \exists a_2...\psi (x,y,a_0,a_1,...) \land \\
\lnot ( \forall a_1 \exists a_2... ( ( x,y ,m, a_0 ,a_1 ,a_2...) \text{ is coded into $\vec{S}$}) )
\end{align*}

\subsection{Forcing global $\Sigma$-uniformization}
This section will make precise the strategy outlined in the previous section.
We will need the following notions. For an arbitrary real $x$, list the triples
$(x,\dot{y}^0,\dot{a}_0^0), (x,\dot{y}^1,\dot{a}_0^1), (x,\dot{y}^2,\dot{a}_0^2),...)$ according to our fixed well-order $<$. To be more precise we list all the $\forceP$-names of reals which are recursive codes for triples of $\forceP$-names of reals, where $\forceP$ is a $\kappa$-sized
partial order. The list should have the property that names of longer iterations always appear after the names of the shorter iterations. We also assume that for every $x$, $\dot{a}_0=\dot{y}_0=0$ to rule out a degenerate case in the following. This has technical advantages as will become clear as we proceed in the argument. The upshot of this is that whenever $(x,a_0=0)$ is in our $\Sigma^1_{n}$-set $A_m$ and the membership is already witnessed by $a_0=0$, then we let $0$ be the $x$-value of our uniformizing function of $A_m$.
Thus we can, without loss of generality, always work under the $\Sigma^1_{n-1}$-assumption that $0$ does not witness that $(x,0)$ is an element of $A_m$. As we aim for a $\Sigma^1_n$-definition of the uniformizing function, this assumption is harmless with regards of complexity.

For each ordinal $\alpha< 2^{\aleph_0}$ we fix
bijections $\pi_{\alpha}: (2^{\aleph_0})^{\alpha} \rightarrow 2^{\aleph_0}$ (we assume w.l.o.g that such a bijection exist as we always force its existence if we blow up the contiuum with, say, Cohen forcing, to size $|\alpha|$ and then use a maximal almost disjoint family $\mathcal{F}$ and almost disjoint coding forcing relative to $\mathcal{F}$ to ensure that in the resulting model $|(2^{\aleph_0})^{\alpha}| = 2^{\aleph_0}$ ). These bijections are of course sensitive to the surrounding universe and we assume that, as we iteratively enlarge our universe via a forcing, the bijections extend each other so that they cohere. To be more precise if $\pi_{\alpha}^{\beta}: (2^{\aleph_0})^{\alpha} \rightarrow 2^{\aleph_0}$ is our chosen family of bijections\footnote{These families of bijections are responsible for the fact that the coding construction presented in this paper and the coding construction which forces the, say, $\Pi^1_3$-uniformization property or the $\Pi^1_3$-reduction property (see \cite{Ho4} and \cite{Ho2}) can not be combined. } in the universe $V[G_{\beta}]$ which arises at the $\beta$-th stage of our iteration,  and $\pi_{\alpha}^{\gamma}: (2^{\aleph_0})^{\alpha} \rightarrow 2^{\aleph_0}$, $\beta< \gamma$ is the  family of bijections we fix at the $\gamma$-th stage $V[G_{\gamma}]$, then for every $\alpha< (2^{\aleph_0})^{V[G_{\beta}]}$,  \[\pi^{\gamma}_{\alpha} \upharpoonright (2^{\omega} \cap V[G_{\beta}] ) = \pi^{\beta}_{\alpha} .\]

Let \[F : \kappa \rightarrow \kappa^{\omega} \] be some bookkeeping function in $L[g^0\ast g^1]$ which shall be our ground model and which guides our iteration. The choice of $F$ does not really matter, it is sufficient to assume that every $x \in \kappa^{\omega}$ has an unbounded pre-image under $F$.
We assume that we have defined already the following list of notions:
\begin{itemize}

\item We defined already our iteration $\forceP_{\beta} \in L[g^0 \ast g^1]$ up to stage $\beta$.
\item We picked a $\forceP_{\beta}$-generic filter $G_{\beta}$ for $\forceP_{\beta}$ and work, as usual, over $L[g^0\ast g^1][G_{\beta}]$.
\item In $L[g^0 \ast g^1] [G_{\beta}]$ we picked a family $\{ \pi_{\alpha} \mid \alpha < 2^{\aleph_0} \}$ of bijections 
of $(2^{\aleph_0})^{\alpha}$ and $2^{\aleph_0}$. These bijections should cohere with the older families of bijections, as discussed above. We assume without loss of generality that such a family exists, if not then we alter $L[g^0 \ast g^1][G_{\beta}]$ with a ccc forcing (namely e.g. an iteration of Cohen forcing followed by an iteration of almost disjoint coding forcing relative to some maximal, almost disjoint family of reals) such that such a family exists in the new universe.
\end{itemize}

We assume  that the bookkeeping function $F$ at $\beta$ hands us an $\omega$ tuple $(\gamma_0,\gamma_1,...)$ of ordinals below $\kappa$ which corresponds to an $ \omega$-tuple $(\dot{x},\dot{m}, \dot{\alpha}, \dot{b}_0,\dot{b}_1,\dot{b}_2,...)$ where $\dot{x}$ and $\dot{b}_n$ are $\forceP_{\beta}$-names of a reals and $\dot{m}$ and $\dot{\alpha}$ are $\forceP_{\beta}$-names of a natural number and of an ordinal  bigger than 0 and less than the size of the continuum respectively  (in fact at each stage we will only need finitely many of those names of reals $\dot{b}_i$ and the rest of the information will be discarded). The correspondence is given by demanding that $\gamma_0 \le \beta$ and  $\dot{x}$ is the $\gamma_1$-th name of a real in $L[g^0][g^1][G_{\gamma_0}]$, $\dot{b}_0$ is the $\gamma_2$-th name of a real in $L[g^0][g^1][G_{\gamma_0}]$ and so forth.

We let $x=\dot{x}^{G_{\beta}}$, and define $b_n,m,\alpha$ and the $G_{\beta}$-evaluation of our list of names of reals $(x,\dot{y}_0,\dot{a}_0),...$ accordingly. 
The natural number $m$ is the G\"odelnumber $\# \varphi$ of a $\Sigma^1_n$-formula $\varphi(x,y)$ in two free variables.
Our goal is to define the forcing $\dot{\forceQ}_{\beta}$ we want to use at stage $\beta$, and to define the notion of $\alpha_{\beta}+1$-based forcing.
We consider various cases for our definition.

\subsubsection{Case 1, odd projective level}

In the first case we write $\varphi_m= \exists a_0 \forall a_1... \exists a_{2n-2} \psi (x,y, a_0,...,a_{2n-2} )$, where $\psi$ is a $\Pi^1_2$-formula (and $\varphi_m$ is a 
$\Sigma^1_{2n+1}$ formula for $n \ge 1$) for the formula with G\"odelnumber $m$ which is handed to us by the value $F(\beta)$.
We repeat our assumption form above that $F(\beta)$ determines
a tuple $b_1,...b_{2n-2}$ of real numbers, $x$, and an ordinal $0<\alpha< 2^{\aleph_0}$ with the assigned reals $a_0^{\alpha}$ and $y^{\alpha}$ which stem from our wellordered list of triples $(x,y^0,a_0^0), (x,y^1, a_0^1),...(x,y^{\alpha},a_0^{\alpha}),...$ \footnote{This list should be injective, that is whenever a triple  $(x,(\dot{y}_{\alpha})^{G_{\beta}},(\dot{a}_{\alpha})^{G_{\beta}})$ appears which appeared already earlier in our list we drop it.} a we fixed in advance.

Recall that we have our fixed bijection $\pi_{\alpha}:( 2^{\aleph_0})^{\alpha} \rightarrow 2^{\aleph_0}$.

Letting $\pi_{\alpha}^{-1} (b_k):= (b^{\eta}_k \mid \eta < \alpha)$ then case 1 is the condition that
\begin{align*}
\lnot \psi (x,y^{\eta},a_0^{\eta}, b_1^{\eta}, b_2^{\eta},...,b_{2n-2}^{\eta} ) \text{ is true for every $\eta < \alpha$.}
\end{align*}

Then we use the coding forcing
\begin{align*}
\operatorname{Code} (\# \psi,x,y^{\alpha},a_0^{\alpha}, b_1,b_2,...,b_{2n-2}) =: \dot{\forceQ}_{\beta}
\end{align*}
as the $\beta$-th factor of our iteration.

\subsubsection{Case 2, odd projective level}
We again let $\varphi_m= \exists a_0 \forall a_1... \exists a_{2n-2} \psi (x,y, a_0,...,a_{2n-2} )$  and we assume that the bookkeeping $F$ at $\beta$ hands us
a tuple $b_1,...b_{2n-2}$ of real numbers, $x$, and $0<\alpha< 2^{\aleph_0}$ with the assigned reals $a_0^{\alpha}$ and $y^{\alpha}$. 

Case 2 is the condition that if $\pi_{\alpha}^{-1} (b_k):= (b^{\eta}_k \mid \eta < \alpha)$ then
\begin{align*}
 \psi (x,y^{\eta},a_0^{\eta}, b_1^{\eta}, b_2^{\eta},...,b_{2n-2}^{\eta} ) \text{ is true for an $\eta < \alpha$.}
\end{align*}

In this situation we force with the trivial forcing at the $\beta$-th stage.

\subsubsection{Case 3, even projective level}
This is the dual case to the first one, but this time $m$ is the G\"odelnumber of a formula which belongs to an even projective level. We write
\[\varphi_m \equiv \exists a_0 \forall a_1... \forall a_{2n-3} (\psi (x,y,a_0,...,a_{2n-3} )),\]
where $\psi$ is a $\Sigma^1_2$-formula and $\varphi_m$ is a $\Sigma^1_{2n}$-formula.

The bookkeeping $F$ at $\beta$ hands us
a tuple $b_1,...,b_{2n-3}$ of real numbers, $x$, and $0<\alpha< 2^{\aleph_0}$ with the assigned reals $a_0^{\alpha}$ and $y^{\alpha}$.

Recall that we have our fixed bijection $\pi_{\alpha}:( 2^{\aleph_0})^{\alpha} \rightarrow 2^{\aleph_0}$.

Case 3 is the condition that if $\pi_{\alpha}^{-1} (b_k):= (b^{\eta}_k \mid \eta < \alpha)$ then
\begin{align*}
\lnot \psi (x,y^{\eta},a_0^{\eta}, b_1^{\eta}, b_2^{\eta},...,b_{2n-3}^{\eta} ) \text{ is true for every $\eta < \alpha$.}
\end{align*}

We do not use a forcing in this case.

\subsubsection{Case 4, even projective level}
This is dual to  case 2. 

We  assume that if $\pi_{\alpha}^{-1} (b_k):= (b^{\eta}_k \mid \eta < \alpha)$ then
\begin{align*}
 \psi (x,y^{\eta},a_0^{\eta}, b_1^{\eta}, b_2^{\eta},...,b_{2n-3}^{\eta} ) \text{ is true for an $\eta < \alpha$.}
\end{align*}
In this situation we use the coding forcing of the form
\[ \operatorname{Code} (\# \varphi_m, x,y^{\alpha}, a_0^{\alpha},b_1,...,b_{2n-3} ) \]
as the $\beta$-th forcing $\dot{\forceQ}_{\beta}$ in our iteration.

\par
\medskip
We use finite support for this iteration and consequently the iteration is a ccc forcing over $L[g^0 \ast g^1]$ hence preserves cofinalities.

This ends the definition of our iteration. What is left is to show that the definition works in that it will produce a universe where $\sigma_{\text{even}}$ and $\sigma_{\text{odd}}$ serve as definitions of uniformizing functions, when applied in the right context.
\subsection{Forcing the continuum being large}
Forcing a large continuum will be achieved by the standard method. 
As in the last section, we assume that we are at stage $\beta < \kappa$ of our iteration and we have defined already $\forceP_{\beta}$. Let $G_{\beta}$ be a $\forceP_{\beta}$ generic filter over $L[g^0 \ast g^1]$ and our fixed bookkeeping function at stage $\beta$ hands us some default set, say $0$. In this situation we will just force with plain Cohen forcing, 
\[ \dot{\forceQ}_{\beta}^{G_{\beta}} := \mathbb{C} \]
and move on to the next step of the iteration.

\subsection{Towards a $\Delta^1_3$-well-order of the reals}
Adding forcings which will eventually produce a $\Delta^1_3$-definable well-order of the reals of the final universe does not cause problems as well. We assume that $\beta< \kappa$, $G_{\beta} \subset \forceP_{\beta}$ a generic filter over $L[g^0][g^1]$ and the bookkeeping $F(\beta)$ yields a pair of names of ordinals $(\dot{\eta}_0,\dot{\eta}_1)$ which evaluate with the help of $G_{\beta}$ to $\eta_0,\eta_1<\kappa$ and we assume that $\eta_0 \le \beta$. We assume that in $L[g^0][g^1][G_{\eta}]$, the $\eta_1$-th (in $L$'s canonical global well-order) pair of $\forceP_{\eta_0}$-names of reals is  $(\dot{b_0},\dot{b_1})$. We proceed as follows. First we evaluate $\dot{b_0}^{G_{\beta}} =b_0$ and $\dot{b_1}^{G_{\beta}}=b_1$ and consider the $<$-least (again, $<$ denotes $L$'s canonical global well-order) forcing names $\sigma_0,\sigma_1$ such that $\sigma_0^{G_{\beta}}=b_0$ and $\sigma_1^{G_{\beta}}=b_1$.
If $\sigma_0 < \sigma_1$ then we use
\[\dot{\forceQ}_{\beta}^{G_{\beta}} := \operatorname{Code} ({b}_0 ,{b_1}), \]
and if $\sigma_1 < \sigma_0$ we let
\[ \dot{\forceQ}_{\beta}^{G_{\beta}} := \operatorname{Code} (b_1,b_0). \]

\section{Properties of the resulting universe}

We shall prove the main properties of our just defined universe now.
Let $G_{\kappa}$ be the final generic filter for our just defined, $\kappa$-length iteration $(\forceP_{\beta},\dot{\forceQ}_{\eta} \mid \beta \le \kappa, \eta < \kappa)$ with finite support. Note that each factor $\dot{\forceQ}_{\beta}$ of our iteration is either an almost disjoint coding forcing of the form $\operatorname{Code} (x)$, for some real $x$, or Cohen forcing. Thus $\forceP_{\kappa}$ has the ccc and all cardinals are preserved. As we will force $\kappa$-many times with Cohen forcing, $\forceP_{\kappa}$ will force that the continuum is $\ge \kappa$.

We now turn to prove that the global $\Sigma$-uniformization property is forced by $\forceP_{\kappa}$.

\begin{lemma}
Let $G_{\kappa}$ denote a generic filter for the full, $\kappa$-length iteration.
\begin{enumerate}

\item
Then $L[g^0][g^1][G_{\kappa}]$ satisfies that
whenever $\varphi_m$ is on an odd projective level, say $\varphi_m= \exists a_0 \forall a_1...\exists a_{2n-2} \psi(x,y,a_0,...,a_{2n-2})$ and $(x,y^{\alpha},a_0^{\alpha})$ is such that
\[ L[g^0][g^1][G_{\kappa}] \models \forall a_1 \exists a_2... \exists a_{2n-2} \psi (x,y^{\alpha},a_0^{\alpha}) \]
Then for each $\beta > \alpha$
\[ L[g^0][g^1][G_{\kappa}] \models \forall a_1 \exists a_2...\exists a_{2n-2} ( (\#\psi,x,y^{\beta} ,a_0^{\beta}, a_1,..,a_{2n-2} ) 
\text{ is not coded }) \]
\item If $\varphi_m= \exists a_0 \forall a_1...\forall a_{2n-3} \psi(x,y,a_0,...,a_{2n-3})$ is a formula on the even projective level and $(x,y^{\alpha},a_0^{\alpha})$ is such that
\[ L[g^0][g^1][G_{\kappa}] \models \varphi_m (x,y^{\alpha},a_0^{\alpha}) \]
Then for each $\beta > \alpha$
\[ L[g^0][g^1][G_{\kappa}] \models \forall a_1 \exists a_2...\forall a_{2n-3} ( (\#\psi, x,y^{\beta} ,a_0^{\beta}, a_1,..,a_{2n-3} ) 
\text{ is coded }) \]
\end{enumerate}

\end{lemma}
\begin{proof}
We will only proof the first item, as the second is shown in the dual way. Let $\beta > \alpha$.
Work in $L[g^0][g^1][G_{\kappa}]$. We know that $\forall a_1 \exists a_2 ... \exists a_{2n-2} \psi (x,y^{\alpha},a_0^{\alpha})$ is true. The way we defined our iteration, in particular case 2, implies that we can use the bijection $\pi_{\beta} : (2^{\omega})^{\beta} \rightarrow 2^{\omega}$, to translate between true infinite disjunctions of the form
\[ \bigvee_{\eta <\beta} \psi (x,y^{\eta}, a^{\eta}_0,b^{\eta}_1,...,b^{\eta}_{2n-2} ) \] and 
``$(\# \psi,x,y^{\beta},a_0^{\beta}, b_1,b_2,...,b_{2n-2} )$ is not coded into $\vec{S}"$, where
$b_i= \pi_{\beta} ((b^{\eta}_i \mid \eta < \beta))$. As $\forall a_1 \exists a_2 ... \exists a_{2n-2} \psi (x,y^{\alpha},a_0^{\alpha})$ is true, this witnesses that our assertion
$\forall a_1 \exists a_2... \exists a_{2n-2}$ $((\#\psi, x,y^{\beta},a^{\beta}_0,a_1,...,a_{2n-2}) $ is not coded into $\vec{S})$ is true in $L[g^0][g^1][G_{\kappa}]$.

\end{proof}

\begin{lemma}
Again $G_{\kappa}$ denotes a generic filter for the entire, $\kappa$-length iteration.
\begin{enumerate}
\item If $\alpha$ is least such that
\[L[g^0][g^1][G_{\kappa}] \models \forall a_1...\exists a_{2n-2} \psi(x,y^{\alpha},a^{\alpha}_0,...,a_{2n-2}),\]
then 
\[ L[g^0][g^1][G_{\kappa}] \models \exists a_1 \forall a_2 ...\forall a_{2n-2} ( (\# \psi, x,y^{\alpha},a_0^{\alpha},a_1,a_2,...,a_{2n-2} ) \text{ is coded}.) \]
\item If $\alpha$ is least such that
\[L[[g^0][g^1]G_{\kappa}] \models \forall a_1...\exists a_{2n-3} \psi(x,y^{\alpha},a^{\alpha}_0,...,a_{2n-3}),\]
then 
\[ L[g^0][g^1][G_{\kappa}] \models \exists a_1 \forall a_2 ...\forall a_{2n-3} ( (\# \psi, x,y^{\alpha},a_0^{\alpha},a_1,a_2,...,a_{2n-3} ) \text{ is not coded}.) \]
\end{enumerate}
\end{lemma}
\begin{proof}
We shall show only the first item of the lemma, as the second one is proved in the dual way.
As $\alpha$ is the first ordinal for which \[L[[g^0][g^1]G_{\kappa}] \models \forall a_1...\exists a_{2n-2} \psi(x,y^{\alpha},a^{\alpha}_0,...,a_{2n-2}),\] we know that
$\forall \beta < \alpha$, $L[g^0][g^1][G_{\kappa}] \models \exists a_1 \forall a_2,...,\forall a_{2n-2} (\lnot \psi (x,y^{\beta},a_0^{\beta}, a_1,...))$.
In particular, for every $\beta < \alpha$ there is a real $a_1^{\beta}$ such that for every $a_2^{\beta}$ there is a real $a_3^{\beta}$ and so on such that $\lnot \psi$ holds. Using the bijection $\pi_{\alpha}$, we can find an $a_1= \pi_{\alpha} ( (a_1^{\beta})_{\beta < \alpha} )$ which has the property that for any real $a_2$ there will be a real $a_3 = \pi_{\alpha} (a_3^{\beta})_{\beta <\alpha} )$ such that for any real $a_4$ and so on $\lnot \psi$ is true. 

But this translates via the case 1 rule of our iteration to the assertion that
\[ L[g^0][g^1][G_{\kappa}] \models \exists a_1 \forall a_2 ...\forall a_{2n-2} ( (\# \psi, x,y^{\alpha},a_0^{\alpha},a_1,a_2,...,a_{2n-2} ) \text{ is coded}.)\]
This is what we wanted.
\end{proof}

\begin{lemma}
In $L[g^0][g^1][G_{\kappa}]$ the $\Sigma^1_{n+1}$-uniformization property holds true for every $n \ge 1$.
\end{lemma}
\begin{proof}
Again we will only consider the case for the odd projective levels. Let $\varphi \equiv \exists a_0 \forall a_1 ...\exists a_{2n-2} (\psi (x,y,a_0,a_1,...,a_{2n-2}))$ be an arbitrary $\Sigma^1_{2n+1}$-formula in two free variables where $\psi$ is $\Pi^1_2$.
Let $x$ be a real such that there is a real $y $ with $L[G_{\kappa}] \models \varphi(x,y)$. We list all the triples $(x,y^{\alpha} ,a_0^{\alpha})$ according to our well-order $<$.
Remember that we defined $y^0, a_0^0$ to be both $0$. If $a_0^0$ witnesses that $\forall a_1 \exists a_2... (\psi (x,0,0,a_1,...)$ is true, which is a $\Pi^1_{2n}$-formula, then $\varphi (x,0)$ will be true in $L[g^0][g^1][G_{\kappa}]$ and is the value of our uniformizing function there.

If $a^0_0=0$ does not witness that $\forall a_1 \exists a_2...(\psi(x,0,0,a_1,a_2,...)$ is true, which is a $\Sigma^1_{2n}$-formula then 
let $\alpha> 0$ be least such that
\[ L[g^0][g^1][G_{\kappa}] \models \forall a_1... (\psi (x,y^{\alpha},a_0^{\alpha},a_1,...) ).\]
Then by the last Lemma
\[ L[g^0][g^1][G_{\kappa}] \models \exists a_1 \forall a_2... \forall a_{2n-2} ( (\# \psi,x,y^{\alpha},a^{\alpha}_0,...) \text{ is coded} ) \]
holds true. Note that this formula is $\Sigma^1_{2n+1}$.

Yet, by the penultimate Lemma, for each $\beta > \alpha$
\[ L[g^0][g^1][G_{\kappa}] \models\forall a_1 \exists a_2... \exists a_{2n-2} ((\#\psi,x,y^{\beta},a^{\beta}_0,a_1,..) \text{ is not coded}). \]

So $(x,y^{\alpha})$ is the unique pair satisfying the $\Sigma^1_{2n+1}$-formula
\begin{align*}
\exists a_0 ( (\forall a_1 \exists a_2...\exists a_{2n-2} ( \psi (x,y,a_0,...) \land \\
\lnot (\forall a_1 \exists a_2 ...\exists a_{2n-2} ((\#\psi,x,y,a_0,a_1,...a_{2n-2} \text{ is not coded} ) )
\end{align*}
Indeed, if $\beta > \alpha$, then the triple $(x,y^{\beta}, a_0^{\beta})$ can not satisfy the second sub-formula above. And if $\beta < \alpha$, then $(x,y^{\beta}, a_0^{\beta} )$ 
will not satisfy \[\forall a_1 \exists a_2...\exists a_{2n-2} \psi( x,y^{\beta}, a_0^{\beta}, a_1,...,a_{2n-2} )\] as $\alpha$ was chosen to be least.

So to summarize the $\Sigma^1_{2n+1}$-formula 
\begin{align*}
&\forall a_1 \exists a_2... (\psi (x,0,0,a_1,a_2,...) \lor \\&
(\exists a_1 \forall a_2...(\lnot \psi (x,0,0,a_1,a_2,...) \land
\exists a_0 ( (\forall a_1 \exists a_2...\exists a_{2n-2} ( \psi (x,y,a_0,...) \land \\&
\lnot (\forall a_1 \exists a_2 ...\exists a_{2n-2} ((\#\psi,x,y,a_0,a_1,...a_{2n-2} \text{ is not coded} ) ) )
\end{align*}
defines the uniformizing function for $\varphi$
\end{proof}

The next lemma is an immediate consequence of the fact that we define our well-order of the reals using the global $L$-well-order. In particular we will never be in a situation where at one stage $\beta_0$ of our iteration, we have $\dot{b}_0^{G_{\beta_0}} < \dot{b}_1^{G_{\beta_0}}$ and at a later stage $\beta_1 > \beta_0$ we have 
$\dot{b}_1^{G_{\beta_0}} > \dot{b}_0^{G_{\beta_0}}$. Thus, given two arbitrary reals $b_0$ and $b_1$ in $L [g^0][g^1][G_{\kappa}]$, either the pair $(b_0,b_1)$ got coded into $\vec{S}$ or the pair $(b_1,b_0)$ got coded into $\vec{S}$.
\begin{lemma}
In $L[g^0][g^1][G_{\kappa}]$ the reals have a $\Delta^1_3$-definable well-order via
\[ \text{$b_0 < b_1$ iff $(b_0,b_1)$ is coded into $\vec{S}$} \]
and
\[  \text{$b_1 < b_0$ iff $(b_1,b_0)$ is coded into $\vec{S}$}. \]
\end{lemma}


\begin{thebibliography}{12}
\bibitem{Addison}
J. W. Addison \textit{Some Consequences of the Axiom of Constructibility.} Fundamenta Mathematicae 46, pp. 337-357, 1959.

\bibitem{Fischer Friedman Zdomskyy} V. Fischer, S. D. Friedman and L. Zdomskyy \textit{Projective wellorders and MAD families with large continuum.} Annals of Pure and Applied Logic 162, pp. 853-862, 2011.

\bibitem{Fischer Friedman Khomskii} V. Fischer, S. D. Friedman and Y. Khomskii \textit{Measure, category and projective wellorder.} Journal of Logic and Analysis 6 )pp. 1-25, 2014.

\bibitem{Harrington} L. Harrington {\em Long projective wellorderings.} Annals of Mathematical Logic, 12, pp. 1–24, 1977. 

\bibitem{Friedman Hoffelner} S. D. Friedman and S. Hoffelner \textit{A $\Sigma^1_4$-Well-Order of the Reals with $\text{NS}_{\omega_1}$ saturated.} Journal of Symbolic Logic 84, No. 4, pp. 1466-1483, 2019.   

\bibitem{Ho} S. Hoffelner \textit{$\hbox{NS}_{\omega_1}$ $\Delta_1$-definable and saturated.}  Journal of Symbolic Logic 86 (1), pp. 25 - 59, 2021.
    
 \bibitem{Ho2} S. Hoffelner \textit{Forcing the $\Pi^1_3$-Reduction Property and a Failure of $\Pi^1_3$-Uniformization},  Annals of Pure and Applied Logic 174 (8), 2023.

\bibitem{A Failure of Reduction in the Presence of Separation} S. Hoffelner \textit{A Failure of $\Pi^1_{n+3}$-Reduction in the Presence of $\Sigma^1_{n+3}$-Separation}, Submitted.
 
 
 \bibitem{BPFA and Global Uniformization} S. Hoffelner \textit{The global $\Sigma^1_{n+2}$-Uniformization Property and $\BPFA$} Advances in Mathematics, 2025


\bibitem{Ho4} S. Hoffelner \textit{Forcing the $\Pi^1_n$-Uniformization Property}, submitted. 
\bibitem{Ho1} S. Hoffelner \textit{Forcing Axioms and the Uniformization Property} Annals of Pure and Applied Logic 175 (10), 2024.

\bibitem{Ho3} S. Hoffelner \textit{Forcing the $\Sigma^1_3$-Separation Property.} Journal of Mathematical Logic,  2022.




\bibitem{Jech} T. Jech \textit{Set Theory. Third Millenium Edition.} Springer 2003.
 
\bibitem{Kechris} 
A. Kechris 
\textit{Classical Descriptive Set Theory}.  
Springer 1995.

\bibitem{MS} R. D. Schindler \textit{Set Theory. Exploring Independence and Truth} Springer 2014.



\bibitem{SchindlerSteel} R. D. Schindler and J. Steel \textit{The Core Model Induction.} Accessible at \url{https://ivv5hpp.uni-muenster.de/u/rds/core_model_induction.pdf}

\bibitem{Steel}  J. Steel \textit{$\PFA$ implies $\mathsf{AD}^{L(\mathbb{R} ) } $. } Journal of Symbolic Logic
Vol. 70, No. 4, pp. 1255-1296, 2005.






\end{thebibliography}
\end{document}